\documentclass[reqno, 11pt]{amsart}
\usepackage{csquotes}  %to get correct quotation marks
\usepackage{amsmath, amssymb}
\usepackage[margin=1in]{geometry}

\usepackage{multirow}

\usepackage{tikz}
\usetikzlibrary{matrix,shapes,arrows,positioning,chains}

\numberwithin{equation}{section}
\newtheorem{thm}{Theorem}[section]
\newtheorem{lem}{Lemma}[section]
\newtheorem{rem}{Remark}[section]
\newtheorem{prop}{Proposition}[section]
\newtheorem{defn}{Definition}[section]
\newtheorem{cor}{Corollary}[section]

\newtheorem*{rem*}{Remark}

\newcommand{\NDE}{{\Delta}^{\operatorname{NUM}}}
\begin{document}
%%\ziju{0.15} \zihao{-4}{\heiti\zihao{3}} \tableofcontents
%\setcounter{page}{1}
%\pagestyle{plain}

\title[Cahn-Hilliard equations]{Why large time-stepping methods for
the Cahn-Hilliard equation is stable}

\author[D. Li]{Dong Li}
\address[D. Li]
{SUSTech International Center for Mathematics, and Department of Mathematics,  Southern University of Science and Technology, Shenzhen, P.R. China}%
\email{lid@sustech.edu.cn}

\subjclass{35Q35}

\keywords{Cahn-Hilliard, maximum principle}

\begin{abstract}
We consider the Cahn-Hilliard equation with standard double-well potential. We employ a 
prototypical class of first order in time semi-implicit methods with implicit treatment of the linear
dissipation term and explicit extrapolation of the nonlinear term. When the dissipation coefficient
is held small,  a conventional wisdom is to add a judiciously chosen stabilization term
in order to afford relatively large time stepping and speed up the simulation.  In practical numerical
implementations it has been long observed that  the resulting system exhibits remarkable stability 
properties in the regime where the stabilization parameter is $\mathcal O(1)$, the dissipation coefficient
is vanishingly small and the size of the time step is moderately large. In this work we develop a new stability theory to address this perplexing phenomenon.
\end{abstract}
\maketitle
\section{Introduction}%----------------------------introduction
The Cahn-Hilliard equation was introduced in \cite{CH58} to describe the phase separation and
coarsening phenomena (i.e. formation of domains) in binary systems. If $c$ denotes the concentration difference of
the two components, then the Cahn-Hilliard equation can be written as
\begin{align}
\partial_t c =D \Delta \mu = D \Delta ( c^3 -c -\nu \Delta c),
\end{align}
where $D$ is the diffusion coefficient, $\mu$ denotes the chemical potential and
$\sqrt{\nu}$ characterizes the length scale of the transition regions between the domains.
In a typical non-dimensionalized form, we take $D=1$ and rewrite $c$ as $u$. Then 
\begin{align} \label{1}
\begin{cases}
\partial_t u  = \Delta ( f(u) ) - \nu \Delta^2 u,  \quad \boxed{f(u)=u^3-u}, \quad (t,x) \in (0,\infty)\times
\Omega;\\
u \Bigr|_{t=0} =u_0.
\end{cases}
\end{align}
For convenience we take the spatial domain $\Omega$ to be the $2\pi$-periodic
torus $\Omega= [-\pi, \pi]^d$ in physical dimensions $d=1,2,3$. With some adjustments our
analysis can be generalized to other boundary conditions. 
The system \eqref{1} admits a free energy given by
\begin{align}
\mathcal E ( u) = \int_{\Omega} \Bigl( \frac 12 \nu |\nabla u |^2 + \frac 14
(u^2-1)^2 \Bigr) dx. 
\end{align}
For smooth solutions the energy dissipation law takes the form
\begin{align} \label{Eap}
\frac d {dt} \Bigl( \mathcal E (u ) \Bigr)= -\int_{\Omega} | \partial_t \mu |^2 dx.
\quad \mu = u^3-u -\nu \Delta u.
\end{align}
This simple balance relation is quite natural since the system \eqref{1} corresponds to the
gradient flow of $\mathcal E (u)$ in $H^{-1}$.  It is not difficult to check that the average of
$u$ is preserved in time. For convenience we shall tacitly assume $u$ has mean zero in our analysis.
Whilst the a priori control \eqref{Eap} yields strong $H^1$ bounds on the solution, the lack of
maximum principle renders it a nontrivial task to obtain $\mathcal O(1)$ bounds on the maximum
norm of the solution.  In the numerical context this issue turns out to be nontrivial
even in the parabolic setting (cf. \cite{LYZ20, Lia}).
%Even when the original system has strong maximum principle, the numerical
%discretization often generates nontrivial deviations in the max-norm

In the past decades, there has been a lot of progress on designing efficient, accurate and
stable numerical schemes to resolve the plethora of vastly different temporal and spatial
scales in phase field models such as Cahn-Hilliard and Allen-Cahn.
Many powerful numerical methods such as  the convex-splitting scheme \cite{chen2012linear,eyre1998unconditionally,wang2010unconditionally}, the stabilization scheme \cite{shen2010numerical,xu2006stability}, the scalar auxiliary variable (SAV) methods \cite{shen2018scalar}, semi-implicit/implicit-explicit (IMEX) schemes \cite{Lia, Lib, lt2021, SShu17} 
are introduced in order to track accurately the dynamical evolution of the phase field
variable.  However many fundamental questions still remain unsolved concerning the analysis
of these schemes. In this work we consider a class of semi-implicit schemes which were considered
by He, Liu and Tang in \cite{HLT07}.  In a semi-discrete formulation, it reads
\begin{align} \label{semi_e1}
\frac{u^{n+1}-u^n}{\tau} = - \nu \Delta^2 u^{n+1} + A \Delta (u^{n+1}-u^n) +\Delta( f(u^n) ), \quad n\ge 0.
\end{align}
where $\tau>0$ is the time step, and $A>0$ is the coefficient for the $\mathcal 
O(\tau)$ regularization term. In \cite{HLT07}, He, Liu and Tang showed that (see
Theorem 1 therein) if
\begin{align} \label{Abad}
A \ge \max_{x\in \Omega}\{ \frac 12 |u^n(x)|^2
+\frac 1 4 |u^{n+1}(x)+u^n(x)|^2 \} -\frac 12, \quad \forall\, n\ge0,
\end{align}
then $\mathcal E(u^n) \le \mathcal E (u^0)$ for all $n\ge 0$.  Note 
that the condition \eqref{Abad} is not satisfactory since the the RHS depends also on $A$. 
An even more startling observation is that one can even take
relatively large time stepping for moderately large $A$ and miniscule dissipation
coefficient $\nu$. For example (see Table $1$ in \cite{HLT07}), numerically one has the following
list of admissible tuple of ($\nu$, $A$, $\tau_c$) where $\tau_c$ is the maximal
time step for which energy decay holds monotonically in time:
\begin{center}
\begin{tabular}{ |c|c|c|c| } 
\hline
 $\nu$ & $A$ & $\tau_c$ \\
\hline
\multirow{3}{5em}{$\nu=0.01$} & $A=0$ & $\tau_c \approx 0.02$ \\ 
& $A=0.5$ & $\tau_c\approx 0.2$ \\ 
& $A=1$ & $\tau_c \approx 0.2$ \\ 
\hline
\multirow{3}{5em}{$\nu=0.001$} & $A=0$ & $\tau_c\approx 0.003$  \\ 
& $A=0.5$ & $\tau_c \approx 0.013$ \\ 
& $A=1$ & $\tau_c \approx 0.03$ \\ 
\hline
\end{tabular}
\end{center}
In particular, for $\nu=0.001$, $A=1$, one can take large time step $\tau\approx 0.03$ whilst
not losing energy dissipation! As far as we know, no existing theory can address this rather perplexing phenomenon.  The purpose of this work is to develop a new stability theory to clarify this
issue. Our first result reveals a deep connection between the stabilization parameter $A$
and the maximum norm of the numerical solution.

\begin{thm}[Uniform in time $L^{\infty}$ bound for \eqref{semi_e1}] \label{thm_bound1}
Assume $A\ge A_{cr}:= 1+ 2\sqrt{1+\frac 43\cdot \frac{\nu}{\tau}}$.
If the initial data $u^0$ satisfies $\|u^0\|_{\infty} \le M=\sqrt{\frac{A+1}3}$, then in
\eqref{semi_e1},
\begin{align*}
\|u^n\|_{\infty} \le M, \qquad\forall\, n \ge 1.
\end{align*}
\end{thm}

\begin{rem}
Note that the threshold value $A_{cr}$ is \emph{not} inversely proportional to the diffusion coefficient
$\nu$. Our $L^{\infty}$ bound here explains that why \eqref{semi_e1} is stable when
the diffusion coefficient $\nu$ is small and large time step $\tau$ is taken.
For example, if $\nu=0.001$ and $\tau=0.03$, then $A_{cr} \approx 3.04$ which is $\mathcal O(1)$!
Of course some further nontrivial work is needed to achieve the optimal stabilization parameter $A\approx 1$. 
\end{rem}

\begin{rem}
There is some flexibility in choosing the upper bound $M$. See Lemma \ref{lem00} where
one can choose any number $M \in [M_0,M_1]$ with $M_1=\sqrt{\frac{A+1}3}$ and
$M_0=2\sqrt{\frac{1+(\frac 12-\sqrt{\frac 14-\frac 1 {A^2}\cdot \frac{\nu}{\tau}})A}3}$.
\end{rem}

Theorem \ref{thm_bound1} elucidates the appearance of $L^{\infty}$ bound due to time
discretization. On the other hand, in practical numerical computations, the bi-harmonic and
Laplacian operators on the RHS of \eqref{semi_e1} would also have to be computed numerically.
In this situation the $L^{\infty}$ bound on the numerical solution certainly needs to be proved as well.
To keep some generality, we denote the numerical approximation of $\Delta$ by $\Delta^{\operatorname{NUM}}$.
For example, on a  1D uniform mesh with mesh size $\Delta x$,
a function $f$ is represented by numerical sequence $f_i$, and a typical central difference scheme on mesh vertex
$i$ (away from the boundary) takes the form
\begin{align*}
(\Delta^{\operatorname{NUM}} f)_i= \frac{f_{i+1}+f_{i-1}-2 f_i} {(\Delta x)^2}.
\end{align*}
In the literature, $\NDE$ is sometimes called the graph Laplacian as it acts on functions defined
 a discrete graph with suitable weights on the edges. We need some \enquote{stability} property of the graph Laplacian $\NDE$.
This is illustrated by the following definition.

\begin{defn} \label{defn1_sharp}
We say a graph Laplacian $\NDE$ on a graph $X$ obeys a sharp $L^{\infty}$ estimate
if the following hold for any constant $k>0$: for any bounded $f: X\to \mathbb R$, there exists a unique
function $u:\, X \to \mathbb R$
solving the equation
\begin{align} \label{defn1_sharp_e1}
u -k \NDE u = f;
\end{align}
moreover
\begin{align*}
\| u \|_{\infty} \le \|f\|_{\infty}.
\end{align*}
In yet other words, for all $k>0$, we have
\begin{align*}
\| (I-k\NDE)^{-1} f \|_{\infty} \le \|f\|_{\infty}.
\end{align*}
\end{defn}

\begin{rem*}
One can certainly consider a more general operator (not necessarily the graph Laplacian) and introduce
the notion of sharp $L^{\infty}$ estimates in more abstract settings. However we do not pursue this generality here.
\end{rem*}

\begin{rem*}
For $\NDE$ introduced via typically finite difference schemes, one can easily verify the
the solvability of \eqref{defn1_sharp_e1} and the sharp $L^{\infty}$ estimate.
See Section \ref{sec_bound1} for some examples.

\end{rem*}

We now consider the following fully discretized (in both space and time) scheme:

\begin{align} \label{semi_e2}
\frac{u^{n+1}-u^n}{\tau} = - \nu (\NDE)^2 u^{n+1} + A \NDE (u^{n+1}-u^n) +\NDE( f(u^n) ), \quad n\ge 0.
\end{align}

\begin{cor} \label{cor_bound1}
Assume $\NDE$ satisfies the sharp $L^{\infty}$ estimate in the sense of Definition \ref{defn1_sharp}.
Let $A\ge A_{cr}:= 1+ 2\sqrt{1+\frac 43\cdot \frac{\nu}{\tau}}$.
If the initial data $u^0$ satisfies $\|u^0\|_{\infty} \le M=\sqrt{\frac{A+1}3}$, then in
\eqref{semi_e2},
\begin{align*}
\|u^n\|_{\infty} \le M, \qquad\forall\, n \ge 1.
\end{align*}

\end{cor}

We state Corollary \ref{cor_bound1} as a conditional result just to keep some generality.
On the other hand, as was already mentioned earlier,
 the condition on $\NDE$ can be easily checked for typical finite difference
 schemes (see Section \ref{sec_bound1}). The following corollary records this fact.

\begin{cor} \label{cor_bound2}
The graph Laplacian $\NDE$ introduced by typical finite difference schemes
satisfies the sharp $L^{\infty}$ estimate in the sense of Definition \ref{defn1_sharp}.
Therefore Corollary \ref{cor_bound1} holds for \eqref{semi_e2} with corresponding $\NDE$.
\end{cor}

\begin{thm} \label{thm_es_1}
Consider \eqref{semi_e1}. Recall
\begin{align*}
\mathcal E(u) = \int_{\Omega} \Bigl( \frac 12 \nu |\nabla u|^2 +F(u) \Bigr) dx,
\end{align*}
where $F(u)=\frac 14 (u^2-1)^2$. Assume (as in Theorem \ref{thm_bound1}) $A\ge A_{cr}=
1+2\sqrt{1+\frac 43 \cdot \frac{\nu}{\tau}}$ and the initial data $u^0$
satisfies $\|u^0\|_{\infty} \le \sqrt{\frac{A+1}3}$. Then
\begin{align*}
& \mathcal E(u^{n+1}) + \frac A 2 \| u^{n+1}-u^n \|_2^2 \notag \\
& \quad + \tau \| \nabla \bigl(  -\nu \Delta u^{n+1} + A (u^{n+1}-u^n) + f(u^n) \bigr) \|_2^2
\notag \\
& \quad \le \mathcal E(u^n), \qquad \forall\, n\ge 0.
\end{align*}
In particular
\begin{align*}
\mathcal E(u^{n+1}) \le \mathcal E(u^n), \qquad\forall\, n\ge 0.
\end{align*}
\end{thm}

\begin{rem*}
To understand the role of the  stabilization term $A\Delta (u^{n+1}-u^n)$, it is useful
to consider the general case
\begin{align*}
\frac{u^{n+1}-u^n}{\tau} = -\nu \Delta^2 u^{n+1} + B (u^{n+1}-u^n)+\Delta f(u^n),
\end{align*}
where $B$ is an operator to be determined. Taking the $L^2$ inner product with $(-\Delta)^{-1}(u^{n+1}-u^n)$ on both
sides, one arrives at
\begin{align*}
&\frac{1}{\tau} \||\nabla|^{-1}(u^{n+1}-u^n)\|_2^2 + E_{n+1}-E_n+ \frac{\nu}2 \| \nabla (u^{n+1}-u^n) \|_2^2
+ ( B(u^{n+1}-u^n), (-\Delta)^{-1} (u^{n+1}-u^n) ) \notag \\
\le\; & \frac L2 \| u^{n+1}-u^n \|_2^2,
\end{align*}
where $L= \sup_{0\le s\le 1} \| f^{\prime}(u^{n}+s(u^{n+1}-u^n) ) \|_{\infty}$ and we have denoted $E_n =\mathcal E (u^n)$.  It should be noted here the rough estimate of $f^{\prime}$
makes no use of the spectral information around linearization of the continuous PDE solution.
Clearly if $B\equiv 0$,  then
to ensure $E_{n+1}\le E_n$, one must enforce
\begin{align*}
\frac{1}{\tau} \||\nabla|^{-1}(u^{n+1}-u^n)\|_2^2 +\frac{\nu}2 \| \nabla (u^{n+1}-u^n) \|_2^2
\ge \frac L2 \| u^{n+1}-u^n \|_2^2.
\end{align*}
In view of the interpolation inequality (for mean-zero functions)
$$\|g\|_2 \le \| |\nabla|^{-1}g \|_2^{\frac 12} \| \nabla g \|_2^{\frac 12}$$
 and Cauchy-Schwartz, we
deduce the constraint
\begin{align*}
2 \sqrt{\frac {\nu}{2\tau}} \ge \frac L2 \Rightarrow \tau \le \frac{8\nu}{L^2}.
\end{align*}
This is the main reason why small time step $\tau$ is needed when $\nu$ is small and
 no stabilization term is present. On
the other hand, from the above computation, one can also see the necessity of having the operator
 $B = \operatorname{const} \cdot\Delta$:
it is precisely used to balance out the term $\frac L2 \|u^{n+1}-u^n\|_2$ on the RHS.

\end{rem*}

The rest of this paper is organized as follows. In the next section we give the proof of the
main result Theorem \ref{thm_bound1}. In Section 3 we give a resolvent bound. 
In the last section we prove Theorem \ref{thm_es_1}.

\section{Proof of Theorem \ref{thm_bound1}, Corollary \ref{cor_bound1} and \ref{cor_bound2}}
\label{sec_bound1}
\subsection*{Proof of Theorem \ref{thm_bound1}}
Write
\begin{align*}
u^{n+1} - u^n = -\nu \tau \Delta^2 u^{n+1} + A \tau \Delta (u^{n+1}
-u^n) +\tau \Delta (f(u^n)).
\end{align*}

Let $\beta>0$ be a parameter whose value will be chosen later. Then
\begin{align*}
(1-\beta A \tau \Delta) (u^{n+1}-u^n) &=-\nu \tau \Delta^2 u^{n+1}
+(1-\beta)A\tau \Delta (u^{n+1}-u^n) +\tau \Delta (f(u^n))
\notag \\
& = \tau \Delta \bigl( (1-\beta)A -\nu \Delta\bigr) u^{n+1} +\tau
\Delta \bigl( f(u^n) - (1-\beta)A u^n \bigr).
\end{align*}
Now choose $\beta$ such that
\begin{align*}
 \frac 1 {\beta A \tau} = \frac{(1-\beta)A} {\nu}
 \end{align*}
or simply
\begin{align*}
\beta(1-\beta) = \frac {\nu} { A^2 \tau}.
\end{align*}
The existence of $\beta$ is out of question since $\nu/(A^2 \tau)
\le 1/4$ by assumption (see below).

Then clearly
\begin{align*}
&u^{n+1}-u^n \notag \\
 =&(1-\beta)A\tau \Delta u^{n+1} + (1-\beta A \tau
\Delta)^{-1} \tau \Delta \bigl( f(u^n) - (1-\beta)A u^n \bigr).
\end{align*}

Rearranging the terms, we get
\begin{align*}
&\bigl( 1- (1-\beta)A\tau \Delta \bigr) u^{n+1} \notag \\
= & \; u^n + (1-\beta A \tau \Delta)^{-1} \tau \Delta \bigl( f(u^n)
- (1-\beta)A u^n \bigr).
\end{align*}

The proof of Theorem \ref{thm_bound1} then follows from Lemma \ref{lem00} below.

\begin{lem} \label{lem00}
Let $k=\nu/\tau$. Define $A_{cr}=1+2\sqrt{1+\frac43k}$. If $A\ge A_{cr}$,
then the following hold:

\begin{itemize}
\item $\nu/A^2\tau\le \frac 14$ and $\beta=\frac 12 +\sqrt{\frac 14
- \frac 1 {A^2} k} \in[\frac 12, 1)$.
\item Define $M_0=2\sqrt{\frac{1+(1-\beta)A}3}$,
$M_1=\sqrt{\frac{A+1}3}$. Then $M_0\le M_1$.
\item For any $M$ with $M_0\le M\le M_1$,  if $\|u^n\|_{\infty}\le
M$, then
\begin{align*}
\| u^n + \tau \Delta (1-\beta A\tau \Delta)^{-1} \bigl( f(u^n) -
(1-\beta) A u^n \bigr) \|_{\infty} \le M;
\end{align*}
and consequently $\|u^{n+1}\|_{\infty} \le M$.
\end{itemize}
\end{lem}

\begin{rem}
Lemma \ref{lem00} shows that for $k>0$, the nonlocal operator $(1-k\Delta)^{-1} \Delta$ exhibits some
form of maximum principle. Interestingly there exist also some \enquote{inverse Sobolev} type equalities
for this operator, see \cite{LYZ13} for more details.
\end{rem}
To complete the proof of Lemma \ref{lem00}, we need the following
simple lemma which in a sense identifies the \enquote{invariant
region} of certain auxiliary cubic polynomials.

\begin{lem} \label{lem00a}
Let $\alpha_1>0$ and $f_1(x)=x^3-\alpha_1 x$. If $L\ge 2
\sqrt{\frac{\alpha_1}3}$, then
\begin{align} \label{lem00a_e1}
\max_{|x|\le L} |f_1(x)| \le f_1(L).
\end{align}
Similarly let $\alpha_2>0$ and $f_2(x)=-x^3+\alpha_2 x$. If $0<L\le
\sqrt{\frac{\alpha_2}3}$, then
\begin{align} \label{lem00a_e2}
\max_{|x|\le L} |f_2(x)| \le f_2(L).
\end{align}
\end{lem}
\begin{proof}[Proof of Lemma \ref{lem00a}]
For $f_1(x)$, calculating $f_1^{\prime}(x)=0$ yields that $x_1=\pm
\sqrt{\alpha_1/3}$. It is then easy to check that at $L=2x_1$,
$f_1(L)\ge |f_1(x_1)|$. An inspection of the graph of $f_1$ easily
gives \eqref{lem00a_e1}. For  \eqref{lem00a_e2}, one just need to
notice that $f_2^{\prime}\ge 0$ for $x\le \sqrt{\frac{\alpha_2}3}$.

\end{proof}

\begin{proof}[Proof of Lemma \ref{lem00}]
First note that
\begin{align*}
\tau \Delta (1-\beta A\tau \Delta)^{-1} &
= (\tau \Delta - \frac 1 {\beta A} + \frac 1 {\beta A} ) (1-\beta A \tau \Delta)^{-1} \notag \\
& = -\frac 1 {\beta A} + \frac 1 {\beta A} (1-\beta A \tau \Delta)^{-1}.
\end{align*}
Thus
\begin{align}
 & u^n + \tau \Delta (1-\beta A \tau \Delta)^{-1} \bigl( f(u^n) - (1-\beta) A u^n \bigr) \notag \\
 = & u^n - \frac 1 {\beta A} (f(u^n) - (1-\beta)A u^n) + \frac 1 {\beta A}
 (1-\beta A \tau \Delta)^{-1} \bigl( f(u^n) - (1-\beta)A u^n \bigr) \notag \\
 = & \frac 1 {\beta A}
 \Bigl(  \underbrace{-(u^n)^3 + (A+1) u^n}_{:=f_2(u^n)}
 + (1-\beta A \tau \Delta)^{-1}
 \bigl( \underbrace{(u^n)^3 - ((1-\beta)A+1)u^n}_{:=f_1(u^n)} \bigr) \Bigr), \label{lem00_e3}
 \end{align}
where in the last equality above, we plugged in $f(u)=u^3-u$.

By Lemma \ref{lem00a}, we have if $\|u^n\|_{\infty} \le M$, then
\begin{align*}
  &\| (1-\beta A \tau \Delta)^{-1} \bigl( f_1(u^n) \bigr) \|_{\infty} \notag \\
\le & \max_{|z|\le M} |f_1(z)| \le f_1(M),
\end{align*}
provided $M\ge 2 \sqrt{\frac{1+(1-\beta)A}3} $.

Then under the condition $\|u^n\|_{\infty} \le M$ and for $M\le \sqrt{\frac {A+1}3}$ (by using Lemma \ref{lem00a}),
\begin{align*}
\|\text{RHS of \eqref{lem00_e3}}\|_{\infty} & \le \frac 1 {\beta A}\Bigl( \max_{|z|\le M} |f_2(z)| +
 f_1(M) \Bigr) \notag \\
& \le \frac 1 {\beta A} \Bigl( f_2(M) +f_1(M) \Bigr) = M.
\end{align*}

Collecting all the inequalities, we get
\begin{itemize}
\item $\beta(1-\beta) = \frac{\nu}{A^2 \tau} \le \frac 14$, $0<\beta<1$;
\item $2\sqrt{\frac{1+(1-\beta)A} 3} \le \sqrt{\frac {A+1} 3}$.
\end{itemize}

It is then easy to deduce the condition $A\ge A_{cr}$.
\end{proof}

\subsection{Proof of Corollary \ref{cor_bound1} and \ref{cor_bound2}}
We first note that in view of \eqref{defn1_sharp_e1},
 the proof of Corollary \ref{cor_bound1} is a repetition of
that of Theorem \ref{thm_bound1} (with $\Delta$ simply replaced by $\NDE$).
Therefore we only focus on Corollary \ref{cor_bound2}. This amounts to
checking Definition \ref{defn1_sharp} for typical finite difference schemes.
We present several illustrative examples.

\begin{itemize}
\item[Example 1:] 1D central difference with periodic boundary condition.
Let $N\ge 2$ be an integer and $\Delta x>0$.  Let $u=(u_0,\cdots,u_{N-1})$ and define
\begin{align*}
(\NDE u)_i = \frac{u_{i+1} +u_{i-1}-2u_i} { (\Delta x)^2}.
\end{align*}
Here $u_{i+N}=u_i$. With data $f=(f_i)$, we need to examine solvability to the equation
\begin{align} \label{cor_bound2_e001a}
u_i - k (\NDE u)_i =f_i
\end{align}
and prove the estimate
\begin{align} \label{cor_bound2_e001b}
\| u \|_{\infty} \le \|f \|_{\infty}.
\end{align}
First we note that \eqref{cor_bound2_e001b} follows from a simple maximum principle argument:
if $i_1=\operatorname{argmax}(u_i)$, then obviously $(\NDE u)_{i_1}\le 0$, and $u_{i_1}\le f_{i_1}$.
To show existence, we can rewrite \eqref{cor_bound2_e001a} as
\begin{align} \label{cor_bound2_e001c}
u_i = (Tu)_i := \frac {\theta} 2(u_{i-1} + u_{i+1}) + (1-\theta) f_i,
\end{align}
where $\theta = \frac{2k}{2k+(\Delta x)^2} $. Since $0<\theta<1$, easy to check that $T$ is a contraction operator (in $l^{\infty}$-norm)
and
the existence follows from the standard fixed point theorem.\footnote{Actually from \eqref{cor_bound2_e001c}
one can also directly deduce the estimate $\|u\|_{\infty} \le \|f\|_{\infty}$ without appealing to the maximum principle.}

\item[Example 2:] 1D central difference with Dirichlet boundary condition.   This is similar to Example 1
except that the boundary condition is modified to $u_{-1}=u_{N}=0$.  Easy to check that in this case
$\NDE$ still satisfies  Definition \ref{defn1_sharp}.

\item[Example 3:] Graph Laplacian with special weights.  Let $X$ be a finite set with cardinality $|X|=N$.
Without loss of generality we identify $X$ as $\{0,\cdots, N-1\}$.
Let $w_{ij}$, $0\le i,j\le N-1$ be nonnegative numbers such that $w_{ii}=\sum_{j\ne i} w_{ij}$, for all $i$.
For any $u: \, X\to \mathbb R$, define
\begin{align}
(\NDE u)_j =  -w_{ii} u_i+\sum_{j\ne i} w_{ij} u_j. \label{cor_bound2_e001d}
\end{align}
Then $\NDE$ satisfies Definition \ref{defn1_sharp}. Indeed the equation $u- k \NDE u =f$ can be rewritten
as
\begin{align}
u_i = (Tu)_i := \sum_{j\ne i} \frac{kw_{ij}}{1+k w_{ii}} u_j + \frac 1 {1+kw_{ii}} f_i. \label{cor_bound2_e001e}
\end{align}
Easy to check that $\|T(u-v)\|_{\infty} \le \theta \| u-v\|_{\infty}$
with
\begin{align*}
\theta = \max_{1\le i\le N} \frac{kw_{ii}}{1+k w_{ii}} <1.
\end{align*}
The estimate $\|u\|_{\infty} \le \|f\|_{\infty}$ is also obvious.

\begin{rem*}
The above example includes many finite difference schemes as special cases. For example,
on a 2D mesh with mesh size $h$, the usual five-point stencil discretized Laplacian has the
form
\begin{align*}
&(\NDE u)(x_1,x_2) \notag\\
=&\; \frac{ u(x_1-h,x_2)+u(x_1+h,x_2)+u(x_1,x_2-h)+u(x_1,x_2+h)-4u(x_1,x_2)}{ h^2}.
\end{align*}
This certainly can be rewritten in the style of \eqref{cor_bound2_e001d}.

\end{rem*}

\end{itemize}

\section{Improved resolvent bounds}
The resolvent bound $\| (I-k \NDE)^{-1} f \|_{\infty} \le \|f\|_{\infty}$ discussed in the previous
section is generally optimal, as can been seen by taking $f$ to be a constant function. On the other hand,
for Cahn-Hilliard type equations, we usually work with functions with mean zero. As it turns out,
for discretized Laplacians, one can refine the resolvent bound slightly if we restrict to the class of mean-zero functions.

\begin{prop} \label{prop_2.4}
Consider \eqref{cor_bound2_e001c}. There exists a constant $0<\epsilon <1$ (possibly depending on $\theta$
and $N$) such that
\begin{align*}
\|u \|_{\infty} \le \epsilon \|f\|_{\infty},
\end{align*}
for any $f$ with mean zero, i.e. $\sum_{i=0}^{N-1} f_i=0$.
\end{prop}
\begin{rem}
To see why Proposition \ref{prop_2.4} should hold, one can consider the special case $N=3$. In this
case by using $u_0+u_1+u_2=0$, one can explicitly solve $u_i$ in terms of $f_i$ as
\begin{align*}
u_i = \frac{1-\theta}{1+\frac{\theta}2} f_i.
\end{align*}
Obviously
$\|u\|_{\infty} \le \frac{1-\theta}{1+\frac{\theta}2} \|f\|_{\infty}$.
\end{rem}
%\begin{rem*}
%The analogue of Proposition \ref{prop_2.4} does not hold in the continuous setting. For example
%consider the equation
%\begin{align*}
% u(x)- k u^{\prime\prime}(x)=f(x),
% \end{align*}
% on the 1D $2\pi$-periodic domain. One can find a $C^{\infty}$ odd function $u$ such that it takes its maximum
% at some $x_0$ and $u^{\prime\prime}(x_0)=0$. In this case $u$ has mean zero (due to parity) and
% $\|u\|_{\infty} =\|f\|_{\infty}$.
%\end{rem*}

To prove Proposition \ref{prop_2.4}, we need a simple lemma. The subtlety lies in
the incorporation of the mean-zero constraint.

\begin{lem}\label{lem_max_1}
Let $N\ge 2$ be an integer. Suppose $0\le c_0\le c_1\cdots\le c_{N-1}$. Let
\begin{align*}
X= \bigl\{ \sigma=(\sigma_0,\cdots,\sigma_{N-1}):\, \max_{j} |\sigma_j|\le 1, \, \sum_{j=0}^{N-1} \sigma_j=0.
\bigr\}.
\end{align*}
Then
\begin{align*}
\max_{\sigma \in X} (c\cdot \sigma) = \sum_{j=N-[\frac N2]}^{N-1} c_j -\sum_{j=0}^{[\frac N2] -1} c_j.
\end{align*}
Here $[x]$ denotes the integer part of any real number $x$, for example $[3/2]=1$.
\end{lem}

\begin{rem*}
If $N$ is even, then the maximum of $c\cdot \sigma$ is achieved by
\begin{align*}
\sigma=(-1,-1,\cdots,-1,1,\cdots, 1)
\end{align*}
 with equal number of $1$s and $-1$s. If $N$ is odd, then this is achieved by
\begin{align*}
\sigma=(-1,-1,\cdots,-1,0,1,\cdots,1)
\end{align*}
 with $(N-1)/2$ ones and minus ones.
\end{rem*}

\begin{proof}[Proof of Lemma \ref{lem_max_1}]
Consider the function $f(\sigma) = c\cdot \sigma$. Since $X$ is a compact set, the maximum of $f$ must be
attained at some point $\tilde \sigma=(\tilde \sigma_0,\cdots, \tilde \sigma_{N-1})$.
Since $0\le c_0\le \cdots c_{N-1}$ and $\sum_{j} \tilde \sigma_j =0$, we can assume $\tilde \sigma_0\le \cdots \tilde \sigma_{j_1}
\le 0 \le \tilde \sigma_{j_1+1} \le \cdots \le \tilde \sigma_{N-1}$. By a simple
optimization argument,\footnote{One can fix the sum $\sum_{l=0}^{j_1} \tilde \sigma_l$
and maximize $\sum_{l=0}^{j_1} \tilde \sigma_l \cdot c_l$. Similarly fix $\sum_{l=j_1+1}^{N-1} \tilde \sigma_l$
and maximize $\sum_{l={j_1+1}}^{N-1} \sigma_l \cdot c_l$.
Also observe that one can assume without loss of generality that there is at most one zero in $\tilde \sigma$.}
one can further assume that $\tilde \sigma$ has three possible forms:

\begin{itemize}
\item $\tilde \sigma=(-1,\cdots, -1,\sigma_{j_1}, \sigma_{j_1+1},1,\cdots, 1)$, where $-1<\sigma_{j_1}\le 0$ and $0\le \sigma_{j_1+1}<1$.
Now since $c_{j_1}\le c_{j_1+1}$, for $\epsilon>0$, we have
\begin{align*}
c_{j_1} \sigma_{j_1} + c_{j_1+1} \sigma_{j_1+1} \le c_{j_1} (\sigma_{j_1}-\epsilon)
+ c_{j_1+1} (\sigma_{j_1+1}+\epsilon).
\end{align*}
By using this argument together with the fact $\sum_j \tilde \sigma_j =0$,
it is easy to see that we can change $\tilde \sigma$ to $\tilde \sigma=(-1,\cdots,-1,1,\cdots, 1)$ and the value of
$c\cdot \tilde \sigma$ does not decrease.

\item $\tilde \sigma=(-1,\cdots,-1,\tilde \sigma_{j_1},1,\cdots,1)$ where $-1<\tilde \sigma_{j_1}\le 0$. Since
$\sum_j \tilde \sigma_j=0$, easy to see that in this case we must have $\tilde \sigma_{j_1}=0$.

\item $\tilde \sigma=(-1,\cdots,-1,\tilde \sigma_{j_1},1,\cdots,1$ where $0\le \sigma_{j_1}<1$. Easy to see
that $\tilde \sigma_1 =0$ again due to $\sum_{j} \tilde \sigma_j =0$.
\end{itemize}
The rest of the argument is now obvious. One just need to discuss separately the case $N$ is even and the
case $N$ is odd.
\end{proof}

\begin{proof}[Proof of Proposition \ref{prop_2.4}]

\texttt{Step 1}. We first show that there exists $c=(c_0,\cdots,c_{N-1})$, such that
\begin{align*}
u_k = (c* f)_k =\sum_{j} c_{k-j} f_j,
\end{align*}
with the identification that $c_{k\pm N}=c_k$. This follows easily from the discrete Fourier transform, which
we briefly recall here. For a sequence of numbers $a_0,\cdots,a_{N-1}$, define
\begin{align*}
\hat a_j = \sum_{k=0}^{N-1} a_k e^{-\frac{2\pi i jk} N}.
\end{align*}
Then $a_k$ can be reproduced from $\hat a_j$ by
\begin{align*}
a_k = \frac 1 N \sum_{j=0}^{N-1} \hat a_j e^{\frac{2\pi i jk} N}.
\end{align*}
For any two sequences $a=(a_0,\cdots,a_{N-1})$ and $b=(b_0, \cdots, b_{N-1})$, easy to check that
\begin{align*}
(\widehat{a*b})_k = \hat a_k \hat b_k.
\end{align*}
Now return to \eqref{cor_bound2_e001c}. Clearly
\begin{align*}
(1 - \theta \cos(\frac{2\pi k}N)) \hat u_k = \hat f_k.
\end{align*}
Thus
\begin{align*}
u_j = (c*f)_j,
\end{align*}
where
\begin{align*}
c_j =\frac 1N  \sum_{k=0}^{N-1} \frac 1 {1-\theta \cos(\frac{2\pi k}N) } e^{\frac{2\pi ijk}N}.
\end{align*}

\texttt{Step 2}. We show that $\sum_{j=0}^{N-1} c_j=1$ and
\begin{align} \label{pf_prop_2.4_e30a}
\min_{0\le j\le N-1} c_j >0.
\end{align}
By Step 1, if we solve
\begin{align} \label{pf_prop_2.4_e30}
u_j = \frac {\theta} 2 (u_{j-1} +u_{j+1}) + (1-\theta) f_j,
\end{align}
with $f=(1,0,\cdots,0)$. Then $u_j=c_{j-1}$. By a simple maximum principle argument we have $u_j\ge 0$
for all $j$. Now assume $u_{j_*}=0$ for some $j_*$. Then from \eqref{pf_prop_2.4_e30} evaluated at $j=j_*$,
we get $u_{j_*-1}=u_{j_*+1}=0$. Iterating this argument a couple of times, we get $u_j=0$ for all $j$ which
is obviously impossible. Thus $\min u_j >0$ and \eqref{pf_prop_2.4_e30a} holds. The fact $\sum_j u_j=1$ is
obvious from summing $j$ on both sides of \eqref{pf_prop_2.4_e30}.

\texttt{Step 3}. Define
\begin{align*}
X= \bigl\{ \tilde f=(\tilde f_0,\cdots,\tilde f_{N-1}):\, \max_j|\tilde f_j|\le 1, \, \sum_j \tilde f_j =0\bigr\}.
\end{align*}
By Lemma \ref{lem_max_1} and Step 2, we have
\begin{align*}
\max_{\tilde f \in X} | c\cdot \tilde f | \le
\begin{cases}
 1- 2\sum_{j=0}^{\frac N2-1} c_j,\quad \text{if $N$ is even}, \\
 1-c_{\frac{N-1}2}- 2\sum_{j=0}^{\frac{N-1}2-1} c_j,\quad \text{if $N$ is odd}.
 \end{cases}
\end{align*}
Thus
\begin{align*}
 \max_{\tilde f \in X} | c\cdot \tilde f | \le
1 - N \min_{j} c_j.
 \end{align*}
Therefore
\begin{align*}
\| c* f \|_{\infty} \le \epsilon \|f\|_{\infty},
\end{align*}
where
\begin{align*}
\epsilon \le 1-N\min_j c_j<1.
\end{align*}

\end{proof}
\begin{rem*}
By Lemma \ref{lem_max_1}, one can get the sharp constant
\begin{align*}
\epsilon =\sum_{j=N-[\frac N2]}^{N-1} c_j -\sum_{j=0}^{[\frac N2] -1} c_j.
\end{align*}
On the other hand, to get the bound $\|c* f \|_{\infty} \le (1-N\min_j c_j) \|f\|_{\infty}$, one
could just argue directly without using Lemma \ref{lem_max_1}. Let $\epsilon_0= \min_j c_j$ and
define $\tilde c_j = c_j -\epsilon_0\ge 0$. Then since $f$ has mean zero, we have $c*f =\tilde c*f$.
Thus
\begin{align*}
\|c*f\|_{\infty} &\le \|\tilde c\|_1 \|f\|_{\infty} \notag \\
&= (1- N \epsilon_0) \|f\|_{\infty}.
\end{align*}
A similar perturbation idea is exploited in recent \cite{L13} to show some generalized Poincar\'e inequalities.
\end{rem*}

We record below the generalization of Proposition \ref{prop_2.4}.
\begin{prop} \label{prop_2.5}
Consider \eqref{cor_bound2_e001e}. There exists a constant $0<\epsilon <1$
 such that
\begin{align*}
\|u \|_{\infty} \le \epsilon \|f\|_{\infty},
\end{align*}
for any $f$ with mean zero.
\end{prop}

\begin{proof}[Proof of Proposition \ref{prop_2.5}]
This is similar to the proof of Proposition \ref{prop_2.4} and we only point out the needed modifications.
First let $\delta_{li}$ be the usual Kronecker delta function and let $c^{(l)}_i$ solves (see \eqref{cor_bound2_e001e})
\begin{align}
c^{(l)}_i= \sum_{j\ne i} \frac{kw_{ij}}{1+k w_{ii}} c^{(l)}_j + \frac 1 {1+kw_{ii}} \delta_{li}. \notag
\end{align}
Then clearly the solution to \eqref{cor_bound2_e001e} can be represented by
\begin{align*}
u_i = \sum_{l} c^{(l)}_i f_l.
\end{align*}
Easy to check that $\epsilon_0= \min_{j,l} c_i^{(l)}>0$. Furthermore  (by taking $f$ to be a constant function) easy
to check that $\sum_{l=0}^{N-1} c^{(l)}_i=1$ for any $i$. Using the fact that $f$ has mean zero,  clearly we have
\begin{align*}
| u_i | =\bigl |    \sum_{l=0}^{N-1}  ( c^{(l)}_i - \epsilon_0) f_l  \bigr|
\le (1-N \epsilon_0) \|f\|_{\infty},
\end{align*}
i.e. $\|u\|_{\infty} \le \epsilon \|f\|_{\infty}$ for  $\epsilon=1-N\epsilon_0<1$.

\end{proof}

\section{proof of Theorem \ref{thm_es_1}}
In this proof we denote by $(,)$ the usual $L^2$ inner product for real-valued functions. Denote
\begin{align*}
H= -\nu \Delta u^{n+1} + A(u^{n+1}-u^n) +f(u^n).
\end{align*}
Here we suppress the notational dependence of $H$ on $n$ for simplicity. The scheme
\eqref{semi_e1} simply reads as
\begin{align*}
\frac{u^{n+1}-u^n} {\tau} = \Delta H.
\end{align*}
Clearly then
\begin{align*}
(u^{n+1}-u^n, H) = \tau (\Delta H, H) = - \tau \| \nabla H\|_2^2.
\end{align*}

We now evaluate $(u^{n+1}-u^n, H)$ by examining the contribution of each term in $H$.
First
\begin{align*}
   & ( u^{n+1}-u^n, -\nu \Delta u^{n+1}) \notag \\
   = \; & -\nu \bigl( (u^{n+1}, \Delta u^{n+1}) - (u^n, \Delta u^{n+1}) \bigr) \notag \\
   = \; & \nu \bigl( \| \nabla u^{n+1} \|_2^2 - ( \nabla u^n, \nabla u^{n+1} ) \bigr) \notag \\
   \ge \; & \nu \bigl(  \frac 12 \|\nabla u^{n+1}\|_2^2 - \frac 1 2 \| \nabla u^n \|_2^2 \bigr).
   \end{align*}
Here we used the simple inequality $a^2+ab \ge \frac 12 a^2 -\frac 12 b^2$ for any $a,b\in \mathbb R$.

Next observe
\begin{align*}
(u^{n+1}-u^n, A (u^{n+1}-u^n) ) = A \| u^{n+1}-u^n \|_2^2.
\end{align*}

Finally
\begin{align*}
( u^{n+1}-u^n, f(u^n) ) = (f(u^n)(u^{n+1}-u^n), 1),
\end{align*}
where $1$ denotes the constant function with value $1$ on $\Omega$.
By the Fundamental Theorem of Calculus, we have
\begin{align*}
    F(u^{n+1}) -F(u^n) & = f(u^n)(u^{n+1}-u^n) +\int_{u^n}^{u^{n+1}} (u^{n+1}-s) f^{\prime}(s) ds \notag \\
    & = f(u^n) (u^{n+1}-u^n) +\int_{u^n}^{u^{n+1}} (u^{n+1}-s) (3s^2-1) ds \notag \\
    &= f(u^n) (u^{n+1}-u^n) +3 \int_{u^n}^{u^{n+1}} (u^{n+1}-s) s^2 ds \notag \\
    & \qquad - \frac 12 (u^{n+1}-u^n)^2.
    \end{align*}
By using Theorem \ref{thm_bound1}, we have $\|u^n\|_{\infty} \le M=\sqrt{\frac{A+1}3}$, $\forall\, n\ge 0$. This gives
\begin{align*}
\left| 3 \int_{u^n}^{u^{n+1}} (u^{n+1}-s)s^2 ds \right|
\le \frac 32 |u^{n+1}-u^n|^2 \cdot M^2.
\end{align*}
Thus
\begin{align*}
& (u^{n+1}-u^n, f(u^n) ) \notag \\
\ge \; & \int_{\Omega} F(u^{n+1}) dx - \int_{\Omega} F(u^n) dx + \frac 12 \| u^{n+1}-u^n \|_2^2 \notag \\
& \qquad -\frac 32 \| u^{n+1}-u^n \|_2^2 \cdot M^2.
\end{align*}

Collecting all the estimates, we get
\begin{align*}
 & (u^{n+1}-u^n, H) \notag \\
 \ge \; & \mathcal E(u^{n+1}) -\mathcal E(u^n) +(A+\frac 12 - \frac 32 M^2) \| u^{n+1}-u^n\|_2^2 \notag \\
 =\; & \mathcal E(u^{n+1}) -\mathcal E(u^n) + \frac A 2 \| u^{n+1}-u^n \|_2^2.
 \end{align*}

The desired inequality then follows easily.

\frenchspacing
\bibliographystyle{plain}

\end{document}